\documentclass[12pt]{amsart}
\usepackage{amsmath,amsthm,amsfonts}
\usepackage{hyperref}
\usepackage{enumerate}
\usepackage{hyperref}
\usepackage{url}

\usepackage[shortlabels]{enumitem}
\setlist[enumerate,1]{label={\upshape(\roman*)}}

\newtheorem{theorem}{Theorem}
\newtheorem{lemma}[theorem]{Lemma}

\theoremstyle{definition}
\newtheorem{definition}[theorem]{Definition}

\newtheorem{remark}[theorem]{Remark}

\newcommand{\fX}{\mathfrak{X}}

\newcommand{\allone}{\mathbf{1}}

\title
[Extended double covers]
{Extended double covers of non-symmetric association schemes of class $2$}
\author{Takuya Ikuta}
\address{Kobe Gakuin University, Kobe, 650-8586, Japan}
\email{ikuta@law.kobegakuin.ac.jp}
\author{Akihiro Munemasa}
\address{Tohoku University, Sendai, 980-8579, Japan}
\email{munemasa@math.is.tohoku.ac.jp}
\thanks{This work was supported by JSPS KAKENHI grant number 20K03527.}

\date{February 18, 2022}

\dedicatory{Dedicated to the memory of Igor Farad\v{z}ev}

\begin{document}

\begin{abstract}
In this paper, we give a  method to construct 
non-symmetric association schemes of class $3$
from non-symmetric association schemes of class $2$.
This construction is a non-symmetric analogue of 
the construction of Taylor graphs as an antipodal
double cover of a complete graph.
We also mention how our construction interact with
doubling introduced by Pasechnik.

\end{abstract}
\maketitle

\section{Introduction}\label{intro}

The first eigenmatrix of a non-symmetric association scheme of class $2$ on $m$ points
is given by
\begin{equation}\label{eq:nonsymm}
\begin{bmatrix}
1&\frac{m-1}{2}&\frac{m-1}{2}\\
1&\frac{-1+\sqrt{-m}}{2}&\frac{-1-\sqrt{-m}}{2}\\
1&\frac{-1-\sqrt{-m}}{2}&\frac{-1+\sqrt{-m}}{2}
\end{bmatrix},
\end{equation}
where $m\equiv 3\pmod 4$ (see, for example \cite{NS2012}). 
The digraphs defined by a nontrivial relation of such an
association scheme are known as doubly regular tournaments,
and the existence of such a digraph is equivalent to that of
a skew Hadamard matrix. See \cite{KS} for details.

In this paper, it is shown that a non-symmetric
association scheme with the first eigenmatrix \eqref{eq:nonsymm}
gives rise to an association scheme of class $3$
with the first eigenmatrix 
\begin{equation}\label{PP1}
\begin{bmatrix}
 1&m&m&1 \\
 1&\sqrt{-m}&-\sqrt{-m}&-1 \\
 1&-\sqrt{-m}&\sqrt{-m} &-1\\
 1&-1&-1&1
\end{bmatrix}.
\end{equation}
We call the resulting association scheme with
the first eigenmatrix \eqref{PP1}, the 
\emph{extended double cover} of the original association
scheme of class $2$.
The construction is analogous to that of Taylor graphs
(see \cite[Sect.~1.5]{BCN}). In fact, the association scheme
defined by a Taylor graph has the first eigenmatrix
\begin{equation}\label{taylor}
\begin{bmatrix}
 1&m&m&1 \\
 1&\sqrt{m}&-\sqrt{m}&-1 \\
 1&-\sqrt{m}&\sqrt{m} &-1\\
 1&-1&-1&1
\end{bmatrix}.
\end{equation}

The extended double cover
is a non-symmetric fission of
a cocktail party graph, and self-dual.
According to S.~Y.~Song \cite[(5.3) Lemma]{S},
a self-dual non-symmetric fission of
a complete multipartite graph has
the first eigenmatrix
\[\begin{bmatrix}
1&k_1&k_1&k_2\\
1&\sqrt{-k_1}&-\sqrt{-k_1}&-1\\
1&-\sqrt{-k_1}&\sqrt{-k_1}&-1\\
1&-\frac{k_2+1}{2}&-\frac{k_2+1}{2}&k_2
\end{bmatrix},
\]
and our association scheme with the first eigenmatrix \eqref{PP1} is a
special case where $k_2=1$.

According to I.~A.~Farad\v{z}ev, M.~H.~Klin, and M.~E.~Muzichuk  
\cite[Theorem 2.6.6]{FMM}, D.~Pasechnik invented a construction 
of a non-symmetric association scheme of  class $2$ on $2m+1$ points
with the first eigenmatrix
\begin{equation}\label{FKM-P}
\begin{bmatrix}
1&m&m\\
1&\frac{-1+\sqrt{-(2m+1)}}{2}&\frac{-1-\sqrt{-(2m+1)}}{2}\\
1&\frac{-1-\sqrt{-(2m+1)}}{2}&\frac{-1+\sqrt{-(2m+1)}}{2}
\end{bmatrix},
\end{equation}
provided that there exists an association scheme with
first eigenmatrix \eqref{eq:nonsymm}.
This construction leads to the doubling of 
skew-Hadamard matrices (see \cite[Theorem~14]{KS}).
We will give a direct description of the extended double cover
of the doubling in Theorem~\ref{edc-1}.

One may wonder if an association scheme with
first eigenmatrix \eqref{PP1} is related to a skew-Hadamard
matrix. However, by \cite{nonsym3}, such an association
scheme does not contain a (complex) Hadamard matrix in its Bose--Mesner algebra.

The organization of this paper is as follows.
In Section~\ref{preliminaries} we introduce necessary notation
and give useful properties of non-symmetric association schemes
of class $2$.
In Section~\ref{adj} we construct the adjacency matrices 
needed in our main theorem,
and prove their multiplication formulas.
In Section~\ref{main sec} we present our main results.

\section{Preliminaries}\label{preliminaries}

For fundamentals of the theory of association schemes,
we refer the reader to \cite{BI}.
Let $\fX=(X,\{R_i\}_{i=0}^d)$ be a commutative association scheme of class $d$ on $n$ points. 
Let $A_0,A_1,\dots,A_d$ be the adjacency
matrices of $\fX$. The intersection numbers
$p_{i,j}^{\ell}$ are defined by
\[
A_iA_j=\sum_{\ell=0}^d p_{i,j}^{\ell}A_{\ell},
\]
and the intersection matrices 
$\{B_i\}_{i=0}^d$ are defined by $(B_i)_{j,\ell}=p_{i,j}^{\ell}$.
The linear span $\mathcal{A}=\langle A_0,A_1,\dots,A_d\rangle$ is called
the \emph{Bose--Mesner algebra} of $\fX$, and it has
primitive idempotents $E_0=\frac{1}{|X|}J,E_1,\dots,E_d$.
The first eigenmatrix $P=(P_{i,j})_{0\leq i,j\leq d}$ is 
defined by
\[(A_0,A_1,\dots,A_d)=(E_0,E_1,\dots,E_d)P,\]
and $Q=|X|P^{-1}$ is called the second eigenmatrix of $\fX$.
Then we have 
\[|X|(E_0,E_1,\dots,E_d)=(A_0,A_1,\dots,A_d)Q.\]
Let $k_i$ ($i=0,1,\ldots,d$) and $m_i$ ($i=0,1,\ldots,d$) 
be the valencies and the multiplicities of $\fX$, respectively.
Then the intersection numbers $p_{i,j}^{\ell}$ are given by
\begin{equation}\label{pijk}
p_{i,j}^{\ell}=\frac{1}{nk_{\ell}}\sum_{\nu=0}^dm_{\nu}\overline{P_{\nu,i}}\ \overline{P_{\nu,j}}P_{\nu,\ell}
\end{equation}

Now, assume that $\fX$ is a non-symmetric
association scheme of class $2$ on $m$ points.
Then $k_1=k_2=m_1=m_2=(m-1)/2$.
By \eqref{eq:nonsymm} and \eqref{pijk} we have
\begin{align}
A_1^2&=\frac{m-3}{4}A_1+\frac{m+1}{4}A_2, \label{eq:103} \\
A_2^2&=\frac{m+1}{4}A_1+\frac{m-3}{4}A_2, \label{eq:104} \\
A_1A_2&=\frac{m-1}{2}A_0+\frac{m-3}{4}(A_1+A_2). \label{eq:105}
\end{align}
Then by \eqref{eq:103}, \eqref{eq:104}, and \eqref{eq:105} we have
\begin{align}
A_1^2+A_2^2&=\frac{m-1}{2}(A_1+A_2), \label{eq:78-2}\\
J+2A_1A_2&=mA_0+\frac{m-1}{2}(A_1+A_2). \label{eq:78-3}
\end{align}

\section{Construction of adjacency matrices}\label{adj}

\begin{definition}\label{dfn:extD}
Let $\fX$ be a non-symmetric association scheme of class $2$ on $m$ points, and
denote its adjacency matrices by $A_0=I_m,A_1,A_2$.
Define
\begin{align}
C_0&=I_{2(m+1)}, \label{eq:0713-0} \\
C_1&=\begin{bmatrix}
  0&\allone&0&0 \\
  0&A_1&A_2&\allone^\top \\
  \allone^\top&A_2&A_1&0 \\
  0&0&\allone&0
  \end{bmatrix}, \label{eq:0713-1}\\
 C_2&=C_1^\top, \label{eq:0713-2}\\
 C_3&=J-C_0-C_1-C_2, \label{C3}
\end{align}
where $\allone$ is the all-one row vector of length $m$.
The association scheme defined by the set of adjacency matrices
$\{C_0,C_1,C_2,C_3\}$ 
is called
the \emph{extended double cover} of $\fX$.
\end{definition}

We will show in the next section that 
$\{C_0, C_1, C_2, C_3\}$ is indeed 
the set of adjacency matrices of an association scheme of class $3$.

By 
\eqref{eq:0713-0}--\eqref{C3}, we have
\begin{align}
C_2&=\begin{bmatrix}
 0&0&\allone&0 \\
 \allone^\top&A_2&A_1&0 \\
 0&A_1&A_2&\allone^\top \\
 0&\allone&0&0
 \end{bmatrix}, \label{eq:0713-3}\\
C_3
&=\begin{bmatrix}
 0&0&0&1 \\
 0&0&I_m&0\\
 0&I_m&0&0 \\
 1&0&0&0
\end{bmatrix}. \label{eq:0713-4}
\end{align}

\begin{remark}\label{reverse}
In Definition~\ref{dfn:extD}, if $A_1$ and $A_2$ define
isomorphic digraphs, then it is easy to see that $C_1$ and $C_2$
define isomorphic digraphs.
\end{remark}

The next lemma is necessary in the next section in order
to establish our main result.

\begin{lemma}\label{lem:cal}
We have the following:
\begin{align}
C_1^2&=C_2^2=\frac{m-1}{2}(C_1+C_2)+mC_3, \label{eq:0713-4}\\
C_1C_2&=C_2C_1=mC_0+\frac{m-1}{2}(C_1+ C_2), \label{eq:0713-5}\\
C_1C_3&=C_3C_1=C_2, \label{eq:0713-6}\\
C_2C_3&=C_3C_2= C_1, \label{eq:0713-7}\\
C_3^2&=C_0. \label{eq:0713-8}
\end{align}
\end{lemma}
\begin{proof}
We can easily see that \eqref{eq:0713-6} and \eqref{eq:0713-8} hold.
Then \eqref{eq:0713-7} follows immediately from \eqref{eq:0713-6} and \eqref{eq:0713-8}.

Since
\begin{align*}
(C_1^2)_{1,1}&=(C_2^2)_{4,4}=0, \displaybreak[0]\\
(C_1^2)_{1,4}&=(C_2^2)_{4,1}=m, \displaybreak[0]\\
(C_1^2)_{1,2}&=(C_2^2)_{1,3}=(C_2^2)_{4,2}=(C_2^2)_{4,3}=\frac{m-1}{2}\allone, \displaybreak[0]\\
(C_1^2)_{2,1}&=(C_2^2)_{2,4}=(C_2^2)_{3,1}=(C_2^2)_{3,4}=\frac{m-1}{2}\allone^\top, \displaybreak[0]\\
(C_1^2)_{2,2}&=(C_1^2)_{3,3}=A_1^2+A_2^2=\frac{m-1}{2}(A_1+A_2) &&\text{(by \eqref{eq:78-2})},\displaybreak[0]\\
(C_1^2)_{3,2}&=(C_2^2)_{2,3}=J+2A_1A_2=mI_m+\frac{m-1}{2}(A_1+A_2) &&\text{(by \eqref{eq:78-3})},\displaybreak[0]
\end{align*}
we have $C_1^2=\frac{m-1}{2}(C_1+C_2)+mC_3$.
Similarly, we have $C_2^2=\frac{m-1}{2}(C_1+C_2)+mC_3$.

Finally, \eqref{eq:0713-5} follows from \eqref{eq:0713-4} and \eqref{eq:0713-6}--\eqref{eq:0713-8}.
\end{proof}

\bigskip

\section{Main results}\label{main sec}

\begin{theorem}\label{thm:0514-1}
The extended double cover of
a non-symmetric association scheme 
of class $2$ on $m$ points
is an association scheme with the first eigenmatrix \eqref{PP1}.
\end{theorem}

\begin{proof}
Suppose $\fX$ is a non-symmetric association scheme 
of class $2$ on $m$ points.
Let $C_0,C_1,C_2,C_3$ be the matrices given in
Definition~\ref{dfn:extD}, and 
let $\mathcal{A}=\langle C_0,C_1,C_2,C_3\rangle$ be their
linear span over the field of complex numbers.
First we observe that 
$\mathcal{A}$ is closed under multiplication by Lemma~\ref{lem:cal}.
Thus $\mathcal{A}$ is the Bose--Mesner algebra of 
an association scheme $\tilde{\fX}$ of class $3$.

Secondly we compute the first eigenmatrix of $\tilde{\fX}$.
By Lemma~\ref{lem:cal} the intersection matrices $B_1, B_2, B_3$ of 
$\tilde{\fX}$ are given by
\begin{align}
B_1&=\begin{bmatrix}
 0&1&0&0 \\
 0&\frac{m-1}{2}&\frac{m-1}{2}&m \\
 m&\frac{m-1}{2}&\frac{m-1}{2}&0 \\
 0&0&1&0
\end{bmatrix}, \label{B1}\\
B_2&=\begin{bmatrix}
 0&0&1&0 \\
 m&\frac{m-1}{2}&\frac{m-1}{2}&0 \\
 0&\frac{m-1}{2}&\frac{m-1}{2}&m \\
 0&1&0&0
\end{bmatrix}, \label{B2}\\
B_3&=\begin{bmatrix}
 0&0&0&1 \\
 0&0&1&0 \\
 0&1&0&0 \\
 1&0&0&0
\end{bmatrix}.\label{B3}
\end{align}
Let
\begin{align}
E_0&=\frac{1}{2(m+1)}(C_0+C_1+C_2+C_3), \label{eq:18-1}\\
E_1&=\frac{1}{4}(C_0-\frac{1}{\sqrt{m}}(C_1-C_2)-C_3), \label{eq:18-2}\\
E_2&=\frac{1}{4}(C_0+\frac{1}{\sqrt{m}}(C_1-C_2)-C_3), \label{eq:18-3}\\
E_3&=\frac{1}{2(m+1)}(mC_0-(C_1+C_2)+mC_3). \label{eq:18-4}
\end{align}
Then by \eqref{B1}--\eqref{B3} we have $E_iE_j=\delta_{i,j}E_i$.
By \eqref{eq:18-1}--\eqref{eq:18-4} the second eigenmatrix of 
$\tilde{\fX}$ is given by
\[
\begin{bmatrix}
 1&\frac{m+1}{2}&\frac{m+1}{2}&m \\
 1&-\frac{m+1}{2\sqrt{m}}&\frac{m+1}{2\sqrt{m}}&-1 \\
 1&\frac{m+1}{2\sqrt{m}}&-\frac{m+1}{2\sqrt{m}}&-1 \\
 1&-\frac{m+1}{2}&-\frac{m+1}{2}&m
\end{bmatrix}.
\]
Then the first eigenmatrix of $\tilde{\fX}$ is given by \eqref{PP1}.
\end{proof}

\begin{remark}
In the database of \cite{Hanaki}, 
as08[6], as16[11], and as24[14] can be constructed by Theorem~\ref{thm:0514-1}.
\end{remark}

\begin{remark}
Let $\{A_0, A_1, A_2\}$ be the set of the adjacency matrices of 
a symmetric association scheme of class $2$ with $k=2\mu$ on $m$ points, and
\begin{align*}
D_0&=I_{2(m+1)},\\
D_1&=\begin{bmatrix}
  0&\allone&0&0 \\
  \allone^\top&A_1&A_2&0 \\
  0^\top&A_2&A_1&\allone^\top \\
  0&0&\allone&0
  \end{bmatrix}, \\
 D_2&=\begin{bmatrix}
  0&0&\allone&0 \\
  0&A_2&A_1&\allone^\top \\
  \allone^\top&A_1&A_2&0 \\
  0&\allone&0&0
  \end{bmatrix},\\
 D_3&=J-D_0-D_1-D_2.
\end{align*}
Then, $D_1$ is the adjacency matrix of 
a Taylor graph (see \cite[Sect.~1.5]{BCN}),
and its first eigenmatrix
is given by \eqref{taylor}.
In this sense, the extended double cover 
can be regarded as a non-symmetric analogue of Taylor graphs.
\end{remark}

The following construction is due to D.~Pasechnik
(announced in \cite{Fa1}; 
see \cite[Theorem 2.6.6]{FMM}).

\begin{definition}\label{doubling}
Let $\fX$ be a non-symmetric association scheme of class $2$ on $m$ points, and
denote its adjacency matrices by $A_0=I_m,A_1,A_2$.
Define
\[
\tilde{A_0}=I_{2m+1}, \ 
\tilde{A_1}=\begin{bmatrix}
 0&\allone&0 \\
  0&A_1&A_2+I_m \\
 \allone^\top&A_2&A_2
\end{bmatrix}, \
\tilde{A_2}=\begin{bmatrix}
 0&0&\allone \\
 \allone^\top&A_2&A_1 \\
 0&A_1+I_m&A_1
\end{bmatrix},
\]
where $A_2^\top=A_1$.
The association scheme with adjacency matrices $\{\tilde{A_0},
\tilde{A_1},\tilde{A_2}\}$ is called the \emph{doubling} of $\fX$.
\end{definition}

In Definition~\ref{doubling}, even if $A_1$ and $A_2$ define
isomorphic digraphs, the matrices $\tilde{A}_1$ and $\tilde{A}_2$
may define non-isomorphic digraphs.

Since the doubling of a 
non-symmetric association scheme of class $2$ on $m$ points
is a non-symmetric association scheme of class $2$ on $2m+1$ points,
we can construct the extended double cover on $4(m+1)$ points by Theorem~\ref{thm:0514-1}.
Then we have the following.

\begin{theorem}\label{edc-1}
Let $\{C_0,C_1,C_2,C_3\}$ be set of the adjacency matrices of 
the extended double cover of a non-symmetric association scheme $\fX$
of class $2$ on $m$ points.
Then the set of adjacency matrices of the
extended double cover of the doubling of $\fX$
is given by $\{C_0',C_1',C_2',C_3'\}$, where
\begin{align}
C_0'&=I_{4(m+1)}, \label{CC0}\\
C_1'&=\begin{bmatrix}
  C_1&C_2+I_{2(m+1)} \\
  C_2+C_3&C_2
  \end{bmatrix}, \label{CC1}\\
C_2'&=\begin{bmatrix}
  C_2&C_1+C_3 \\
  C_1+I_{2(m+1)}& C_1
  \end{bmatrix},\label{CC2}\\
C_3'&=\begin{bmatrix}
  C_3&0 \\
  0& C_3
  \end{bmatrix}.\label{CC3}
\end{align}
\end{theorem}
\begin{proof}
Substituting the adjacency matrices
$\tilde{A_0},\tilde{A_1},\tilde{A_2}$ in Definition~$\ref{doubling}$ 
into \eqref{eq:0713-1} directly,
we have
\[
\left[ \begin{array}{c|ccc|ccc|c}
 0&1&\allone&\allone&0&0&0&0 \\
 \hline
 0&0&\allone&0&0&0&\allone&1 \\
 0&0&A_1&A_2+I_m&\allone^\top&A_2&A_1&\allone^\top \\
 0&\allone^\top&A_2&A_2&0&A_1+I_m&A_1&\allone^\top \\
 \hline
 1&0&0&\allone&0&\allone&0&0 \\
 \allone^\top&\allone^\top&A_2&A_1&0&A_1&A_2+I_m&0 \\
 \allone^\top&0&A_1+I_m&A_1&\allone^\top&A_2&A_2&0 \\
 \hline
 0&0&0&0&1&\allone&\allone&0 \\
\end{array} \right].
\]
Rearranging the rows and columns in the order $[5,6,3,2,1,7,4,8]$,
we obtain
\begin{align*}
&\left[\begin{array}{cccc|cccc}
 0 & \allone & 0 & 0 & 1 & 0 & \allone & 0 \\
 0 & A_1 & A_2 & \allone^\top & \allone^\top & A_2 +I_m & A_1 & 0 \\
 \allone^\top & A_2 & A_1 & 0 & 0 & A_1 & A_2 +I_m & \allone^\top \\
 0 & 0 & \allone & 0 & 0 & \allone & 0 & 1 \\ \hline
 0 & 0 & \allone & 1 & 0 & 0 & \allone & 0 \\
 \allone^\top & A_2 & A_1 +I_m & 0 & \allone^\top & A_2 & A_1 & 0 \\
 0 & A_1 +I_m & A_2 & \allone^\top & 0 & A_1 & A_2 & \allone^\top \\
 1 & \allone & 0 & 0 & 0 & \allone & 0 & 0 
\end{array}\right]
\\&=
\begin{bmatrix}
  C_1&C_2+I_{2(m+1)} \\
  C_2+C_3&C_2
  \end{bmatrix}
\\&=C'_1.
\end{align*}
Hence we have \eqref{CC1}.
By \eqref{eq:0713-3} and \eqref{CC1} we have \eqref{CC2}.
By \eqref{C3}, \eqref{CC0}--\eqref{CC2} we have \eqref{CC3}.
\end{proof}

\bigskip

\subsection*{Acknowledgements.} We would like to thank
the anonymous referee for pointing out the connections to Taylor graphs.

\end{document}